\documentclass[11pt]{amsart}

\usepackage[english]{babel}     
\usepackage[utf8]{inputenc}
\usepackage{amsmath}
\usepackage{amsfonts}
\usepackage{amssymb}
\usepackage{mathrsfs}
\usepackage{stmaryrd}
\usepackage{hyperref}
\usepackage{xcolor}
\usepackage{todonotes}
\usepackage{comment}
\usepackage{enumerate}
\usepackage{tikz-cd}

\newtheorem{theorem}{Theorem}[section]

\newtheorem{lemma}[theorem]{Lemma}

\newtheorem{corollary}[theorem]{Corollary}

\theoremstyle{definition}
\newtheorem{definition}[theorem]{Definition}

\newtheorem{remark}[theorem]{Remark}

\theoremstyle{remark}
\numberwithin{equation}{section}

\let\epsilon\varepsilon

\newcommand{\frakg}			{{\operatorname{\mathfrak{g}}}}
\newcommand{\fraka}			{{\operatorname{\mathfrak{a}}}}

\newcommand{\frakn}			{{\operatorname{\mathfrak{n}}}}

\DeclareMathOperator{\id}{id}

\newcommand{\OmegaA}{\Omega_K}

\newcommand{\Dom}{{\rm Dom}}
\newcommand{\SigmaPluss}{\Sigma^{++}}

\title[A Sobolev estimate for radial $L^p$-multipliers on a class of semi-simple Lie groups]{A Sobolev estimate for radial $L^p$-multipliers on a class of semi-simple Lie groups}

\date{\noindent \today.  MSC2010 keywords: 22D25, 22E30, 46L51.  MC is  supported by the NWO Vidi grant VI.Vidi.192.018 `Non-commutative harmonic analysis and rigidity of operator algebras'.}

\author[Martijn Caspers]{Martijn Caspers}

\address{TU Delft, EWI/DIAM,
	P.O.Box 5031,
	2600 GA Delft,
	The Netherlands}

\email{M.P.T.Caspers@tudelft.nl}

\begin{document}

\maketitle

\begin{abstract}
Let $G$ be a semi-simple Lie group in the Harish-Chandra class with maximal compact subgroup $K$. Let $\OmegaA$ be minus the radial Casimir operator.  Let $\frac{1}{4} \dim(G/K) < S_G < \frac{1}{2} \dim(G/K) , s \in (0,  S_G]$ and   $p \in (1,\infty)$ be   such that
\[
\left| \frac{1}{p} - \frac{1}{2} \right| < \frac{s}{2 S_G}.
\]
Then, there exists a constant $C_{G,s,p} >0$ such that for every $m \in L^\infty(G) \cap L^2(G)$ bi-$K$-invariant with $m \in \Dom(\OmegaA^s)$ and $\OmegaA^s(m) \in L^{2S_G/s}(G)$ we have,
\begin{equation}
\Vert T_m: L^p(\widehat{G}) \rightarrow  L^p(\widehat{G}) \Vert \leq C_{G, s,p} \Vert \OmegaA^s(m) \Vert_{L^{2S_G/s}(G)},
\end{equation}
where $T_m$ is the Fourier multiplier with symbol $m$ acting on the non-commutative $L^p$-space of the group von Neumann algebra of $G$. This gives new examples of $L^p$-Fourier multipliers with decay rates becoming slower when $p$ approximates $2$.
\end{abstract}

\section{Introduction}

 Finding sharp norm estimates of Fourier multipliers is a central and intriguingly delicate theme in Euclidean harmonic analysis. Nowadays several multiplier theorems are known including the celebrated H\"ormander-Mikhlin multiplier theorem that estimates the bound of an $L^p$-multiplier in terms of differentiability and regularity properties of the symbol. This is just one of the many theorems in harmonic analysis and we refer to the monographs \cite{GrafakosClassical}, \cite{GrafakosModern}, \cite{Stein} for a broader treatment of the subject.

Over the past decade there has been an increasing interest in the construction of $L^p$-multipliers on a non-abelian locally compact group $G$. Here the group plays the role of the frequency side. For a symbol $m \in L^\infty( G)$ the central question is for which $1 < p < \infty$  the Fourier multiplier
\[
T_m: L^2( \widehat{G} ) \rightarrow   L^2( \widehat{G} ): \lambda_G(f) \mapsto \lambda_G(mf),
\]
extends to a bounded map $L^p( \widehat{G} ) \rightarrow   L^p( \widehat{G} )$. Here $\lambda_G$ is the left regular representation and $L^p(\widehat{G})$ is the non-commutative $L^p$-space of the group von Neumann algebra $L^\infty(\widehat{G})$ of $G$. Moreover, ideally one would have sharp bounds on the norms of such multipliers and understand their regularity properties.

Recently, a number of such multiplier theorems have been obtained. In the realm of discrete groups Mei and Ricard \cite{MeiRicard} gave a free analogue of the Hilbert transform through Cotlar's identity yielding multipliers on free groups. The techniques were exploited further in \cite{MeiRicardXu}, \cite{GonzalezCotlarPaper} for free (amalgamated) products of groups and  groups acting on tree-like structures. In \cite{JungeMeiParcet} (see further \cite{GonzalezSpectralPaper}, \cite{JungeMeiParcetJEMS}) a very effective method based on cocycles was introduced to construct a wide class of multipliers. Some of the results we mentioned so far also yield multipliers for non-discrete groups.

In the case of semi-simple Lie groups the main achievements that have been made are contained in \cite{PRS} and \cite{Tablate}. In particular for ${\rm SL}(n, \mathbb{R})$, the group of $n \times n$-matrices over $\mathbb{R}$ with determinant 1, an analogue of the classical H\"ormander-Mikhlin multiplier theorem is obtained \cite{PRS}. The differentiability properties of the symbol are then described in terms of Lie derivatives. The methods from \cite{PRS}, \cite{Tablate} rely on a local transference from Schur multipliers and local approximations of the Lie group with Euclidean spaces.  The main theorem \cite[Theorem A]{PRS} is extremely effective and close to sharp for symbols that are supported on a small enough neighbourhood of the identity. Moreover, these H\"ormander-Mikhlin type conditions automatically imply certain integrability  of the symbol and therefore the theorem in fact extends from a small neighbourhood to symbols on the entire Lie group. Such symbols thus have a fast decay; fast enough to assure integrability properties of the symbol \cite[Remark 3.8]{PRS}.  Concerning the behavior of (radial) multipliers away from the identity the rigidity theorem \cite[Theorem B]{PRS} shows that symbols of $L^p$-multipliers in fact must necessarily have a sufficient amount of decay. The fundament of this phenomenon stems from  \cite{LafforgueDeLaSalle}, \cite{Laat}, \cite{DeLaat}. In the degree of decay there is a gap between the multiplier theorems  \cite[Theorem A]{PRS}, \cite[Theorem A2]{Tablate} and the rigidity theorem \cite[Theorem B]{PRS} in case $p$ is close to 2. The current paper provides a  new viewpoint on this gap.

We also mention that connections between discrete and locally compact groups have been made through noncommutative versions of De Leeuw theorems \cite{CPPR}, \cite{CJKM}.


\vspace{0.3cm}

The main result of this paper obtains a new multiplier theorem that is applicable to a natural class of semi-simple Lie groups. Our main result is the following theorem as announced in the abstract. The theorem provides a Sobolev type estimate for Fourier multipliers. The regularity properties are formulated in terms of the distance of $p$ from 2.

\begin{theorem} \label{Thm=MainIntro}
Let $G$ be a semi-simple Lie group in the Harish-Chandra class with maximal compact subgroup $K$. Let $\OmegaA$ be minus the radial Casimir operator.  Let $\frac{1}{4} \dim(G/K) < S_G < \frac{1}{2} \dim(G/K) , s \in (0,  S_G]$ and   $p \in (1,\infty)$ be   such that
\[
\left| \frac{1}{p} - \frac{1}{2} \right| < \frac{s}{2 S_G}.
\]
Then, there exists a constant $C_{G,s,p} >0$ such that for every $m \in L^\infty(G) \cap L^2(G)$ bi-$K$-invariant with $m \in \Dom(\OmegaA^s)$ and $\OmegaA^s(m) \in L^{2S_G/s}(G)$ we have,
\begin{equation}\label{Eqn=IntroMain}
\Vert T_m: L^p(\widehat{G}) \rightarrow  L^p(\widehat{G}) \Vert \leq C_{G, s,p} \Vert \OmegaA^s(m) \Vert_{L^{2S_G/s}(G)}.
\end{equation}
\end{theorem}

   There are several novelties in our approach compared to earlier multiplier theorems on non-abelian groups, in particular Lie groups. Firstly it is the first time that differentiability properties with respect to the Casimir operator are used in estimates on $L^p$-multipliers. As the Casimir operator equals the Laplace-Beltrami operator on the homogeneous space $G \slash K$ the estimate \eqref{Eqn=IntroMain} should be understood as a Sobolev norm estimate.
   Secondly, our proof uses the representation theory of $G$ to construct $L^p$-multipliers. In this case we use the spherical  dual of the group to construct bi-$K$-invariant multipliers. This makes a link to Harish-Chandra's Plancherel theorem for spherical functions.  Most notably, we establish a link between the construction of $L^p$-multipliers and Heat kernel estimates of the Casimir operator which are well-studied in the literature. For the Heat kernel estimates and consequent estimates for Bessel-Green-Riesz potentials we shall mostly rely ourselves on Anker-Ji \cite{AnkerJi} (see Section \ref{Sect=HeatKernel}); though in many cases such estimates were obtained earlier, see \cite{GangolliActa}, \cite{AnkerOldPaper}, or for ${\rm SL}(n, \mathbb{R})$, see \cite{Sawyer}.

    Theorem \ref{Thm=MainIntro} is very much in the spirit of the Calder\'on-Torchinsky theorem \cite{CT1}, \cite{CT2} in Euclidean analysis, see also the more recent papers \cite{GrafakosIMRN}, \cite{GrafakosCT}, but there are several fundamental differences too. For instance we do not have a Littlewood-Payley theory at our disposal and neither we are able to control the volumes of translations of areas.  Secondly, the classical results on Bessel potentials that are used in \cite{GrafakosCT} (see \cite[Section 6.1.2]{GrafakosModern}) need to be replaced by the much deeper results of Anker and Ji.

We shall show in Section \ref{Sect=Examples} that Theorem \ref{Thm=MainIntro} leads to several new classes of Fourier multipliers that are fundamentally beyond the reach of the earlier theorems from \cite{PRS}, \cite{Tablate}. For instance, the following question seems to be unknown. Suppose we have a symbol $m$ that is smooth in a neighbourhood of the origin of $G$ and of the form $m(k_1 \exp(H) k_2) = e^{-A_p \Vert H \Vert}, k_1, k_2 \in K, H \in \fraka$ outside that neighbourhood for some constant $A_p > 0$. Determine for which $A_p > 0$ and which $1 < p < \infty$ such a symbol can be an $L^p$-Fourier multipler. The rigidity theorem \cite[Theorem B]{PRS} puts a lower bound on $A_p$ whereas
 \cite[Remark 3.8]{PRS}, which is fundamental to \cite[Theorem A]{PRS}, yields  an upper bound on $A_p$. This paper improves on this upper bound as explained in Section \ref{Sect=Examples}, where we also explain that our theorem reaches beyond \cite[Theorem A2]{Tablate}. This gives a negative answer to the question at the end of \cite[Remark 4.2]{Tablate} if one is allowed to differentiate in $p$.

\vspace{0.3cm}

\noindent {\bf Structure of the paper.} Section \ref{Sect=Prelim} contains preliminaries on Lie groups, von Neumann algebras and non-commutative $L^p$-spaces. In Section \ref{Sect=HeatKernel} we prove an important direct consequence of the results of Anker-Ji \cite{AnkerJi}. Section \ref{Sect=Casimir} contains the core of this paper and provides a first estimate for $L^p$-multipliers. Section \ref{Sect=Interpolation}interpolates between $L^p$ and $L^2$ to conclude the main theorem. Finally Section \ref{Sect=Examples} contains important and surprising new examples of multipliers. We also specialize the case of ${\rm SL}(n, \mathbb{R})$.

\vspace{0.3cm}

\noindent {\bf Acknowledgements.} The author wishes to thank Jordy van Velthoven for several useful discussions and references to the literature.

\section{Preliminaries}\label{Sect=Prelim}
 For expressions  $A$ and $B$ we write  $A \approx B$ if there exists two absolute constants $c_1, c_2 > 0$ such that $c_1 A \leq B \leq c_2 A$. We write $A \preceq B$ if only $c_1 A \leq B$. The symbol $\otimes$ will denote the tensor product. In case we take a tensor product of von Neumann algebras we mean the von Neumann algebraic tensor product, i.e. the strong operator topology closure of their vector space tensor product. In case of a tensor product of $L^p$-spaces we mean the $L^p$-norm closure of that vector space tensor product.

\subsection{Lie groups and Lie algebras} For standard references on Lie groups and Lie algebras we refer to \cite{Helgason}, \cite{Knapp}, \cite{Humphreys} and for spherical functions to \cite{HelgasonSpherical} and \cite{JorgensonLang}.
 This paper crucially relies on \cite[Section 4]{AnkerJi} and therefore from this point onwards we assume that {\it $G$ is semi-simple and in the Harish-Chandra class}, see \cite{Knapp}. This includes all semi-simple, connected, linear Lie groups.
  In particular $G$ has finite center and finitely many connected components.  Let $K$ be a maximal compact subgroup of $G$ which is unique up to conjugation.
Let $G = KAN$ be the Iwasawa decomposition where $A$ is abelian and $N$ is nilpotent. Let $G = KAK$ be the Cartan decomposition of $G$. Let $\frakg$, $\fraka$ and $\frakn$ be the Lie algebras of respectively $G$,  $A$ and $N$.

We let $\Sigma \in \fraka^\ast$ be the set of roots and by fixing a positive Weyl chamber $\fraka^+$ we let $\Sigma^+$ be the set of positive roots.  We let $\SigmaPluss$ be the set of positive indivisible roots; recall that this means that $\alpha \in \SigmaPluss$ if $\alpha$ is a positive root and $\frac{1}{2} \alpha$ is not a root.
 If $G$ is complex then $\Sigma^+ = \SigmaPluss$ (see \cite[Section II.8.4]{Humphreys}); otherwise if $\alpha \in \SigmaPluss$ the only scalar multiples of $\alpha$ that are possibly roots are $-2 \alpha, -\alpha, \alpha$, and $2 \alpha$ (see \cite[Exercise III.9]{Humphreys}) .
 Let $m_\alpha$ be the multiplicity of a root.
  Recall that $\dim(N) = \sum_{\alpha \in \Sigma^+} m_\alpha$.
    We also set $\rho = \frac{1}{2} \sum_{\alpha \in \Sigma^+} m_\alpha \alpha$.

\subsection{Killing form}
Let $\langle \cdot, \cdot \rangle$ be the  Killing form on $\frakg$ which is a non-degenerate bilinear form. The Killing form restricts to a positive definite non-degenerate form on $\fraka$. For $H \in \fraka$ we set $\Vert H \Vert = \langle H, H \rangle^{\frac{1}{2}}$. The Killing form linearly identifies the dual $\fraka^\ast$ with $\fraka$ by identifying $H \in \fraka$ with $\alpha_H( \: \cdot \:) := \langle H, \cdot \rangle \in \fraka^\ast$. Under this identification the pairing $\langle \: \cdot \:, \: \cdot \: \rangle$ and corresponding norm are thus defined on $\fraka^\ast$ as well; this also defines the pairing between $\fraka$ and $\fraka^\ast$.

\subsection{Haar measure}
 We denote $\mu_G$ for the Haar meausure of $G$ which is both left- and right-invariant as $G$ is semi-simple hence unimodular.
  The Haar measure decomposes with respect to the Cartan decomposition (see \cite[Theorem I.5.8]{Helgason} or \cite[Eqn. (2.1.5)]{AnkerJi}) as
 \begin{equation}\label{Eqn=CartanIntegral}
 \int_G f(g) d\mu_G(g) =  \vert K \slash M \vert   \int_{K} \int_{\fraka^+} \int_K f(k_1 \exp(H) k_2)  \delta(H) dk_1 dH dk_2,
 \end{equation}
 where $M$ is group of elements in $K$ that commute with $A$ (i.e. the centralizer) and
 \begin{equation}\label{Eqn=CartanIntegral2}
 \delta(H) = \prod_{\alpha \in \Sigma^+}  \sinh^{m_\alpha} \langle \alpha, H \rangle   \approx \prod_{\alpha \in \Sigma^+}  \left(   \frac{ \langle \alpha, H \rangle   }{   1 + \langle \alpha, H \rangle }  \right)^{m_\alpha}  e^{2 \langle \rho, H \rangle}.
 \end{equation}
 The volume of the quotient $ \vert K \slash M \vert$ will not play a very significant role in our paper and can be regarded as a constant.
 We let $(f_1 \ast f_2)(g) = \int_G f_1(h) f_2(h^{-1} g) d\mu_G(h)$ be the convolution product for suitable $\mathbb{C}$-valued functions $f_1$ and $f_2$ on $G$.

\subsection{Function spaces and $K$-invariance}
Let $F$ be either  $C^\infty, C^\infty_c, C, L^p$ so that we mean by  $F(G)$ either $C^\infty(G),  C^\infty_c(G), C(G),  L^p(G)$ which are respectively the functions $G \rightarrow \mathbb{C}$ that are smooth, smooth with compact support, continuous and $p$-integrable with respect to $\mu_G$.
We use the notation $F(K \backslash G), F(G \slash K), F(K \backslash G \slash K)$ to denote the space of functions in $F(G)$ that are $K$-invariant from the left, right or both left and right respectively.

\subsection{Casimir operator}\label{Sect=Casimir}

 Let $\Omega \in U(\frakg)$ be the Casimir element of $G$ where $U(\frakg)$ is the universal enveloping algebra of $\frakg$. $\Omega$ acts as a second order differential operator on  $C^\infty(G)$ (see \cite[Section II.1]{JorgensonLang}) and we let $\OmegaA^0$ be the restriction of $-\Omega$ to $C^\infty(K \backslash G \slash K)$. On the domain
 \[
 \Dom(\OmegaA^0) = \{ f \in C^\infty( K \backslash  G  \slash K) \cap L^2(  K \backslash  G \slash K) \mid \OmegaA(f) \in L^2(   K \backslash G \slash K) \},
 \]
 the operator $\OmegaA^0$ is essentially self-adjoint; in fact its restriction to the Harish-Chandra Schwartz space  is well-known to be essentially self-adjoint. So the closure $\OmegaA$ of $\OmegaA^0$ is an unbounded self-adjoint operator on $L^2( K \backslash G \slash K )$. Then   $\OmegaA$ is {\it positive} and the spectrum of $\OmegaA$ is the interval $[ \langle \rho, \rho \rangle, \infty)$. We may therefore consider fractional powers $\OmegaA^s$ with $s \in \mathbb{R}$; these fractional powers are bounded operators in case $s \leq 0$.

\subsection{Von Neumann algebras} We denote $B(H)$ for the bounded operators on a Hilbert space $H$ and for $B \subseteq B(H)$ we set $B'= \{ x \in B(H) \mid \forall b \in B:  xb = bx \}$ for the commutant. For von Neumann algebras we refer to \cite{Tak1} as a standard work and for non-commutative $L^p$-spaces to \cite{Nelson}, \cite{PisierXu}.

Let $M$ be a semi-finite von Neumann algebra with normal semi-finite faithful trace $\tau$. By $L^p(M, \tau), 1 \leq p < \infty$ we denote the Banach space consisting of all closed densely defined operators $x$ affiliated with $M$ such that $\Vert x \Vert := \tau( \vert x \vert^p  )^{1/p} < \infty$.  The set $M \cap L^p(M, \tau)$ is dense in $L^p(M, \tau)$; so that alternatively $L^p(M, \tau)$ can be defined as the abstract completion of this intersection. We set $L^\infty(M, \tau) = M$.

\subsection{Group von Neumann algebras} Let $G$ be a locally compact unimodular group; in particular any semi-simple Lie group. We let $(\lambda_G(g) f)( h) = f(g^{-1} h)$ and $(\lambda_G'(g) f)( h) = f(hg), g,h \in G$ be the left- and right regular representation of $G$ on $L^2(G)$. Let $L^\infty(\widehat{G}) = \{ \lambda_G(g) \mid   g \in G \}''$ (double commutant of the set) be the left group von Neumann algebra of $G$. For $f \in L^1(G)$ set $\lambda_G(f) = \int_G f(g) \lambda_G(g) d\mu_G(g) \in L^\infty(\widehat{G})$. $L^\infty(\widehat{G})$ can be equipped with the Plancherel trace $\tau_{\widehat{G}}: L^\infty(\widehat{G})^+ \rightarrow [0, \infty]$ given by
\[
\tau_{\widehat{G}}(x^\ast x) = \Vert f \Vert_2^2,
\]
in case there is $f \in L^2(G)$ such that $x h = f \ast h$ for all $h \in C_c(G)$. We set  $\tau_{\widehat{G}}(x^\ast x) = \infty$ otherwise. Briefly set $L^p(\widehat{G}) = L^p( L^\infty(\widehat{G}), \tau_{\widehat{G}})$.
 Then, by definition,   $\lambda_G$ extends to an isometry $L^2(G) \rightarrow L^2(\widehat{G})$ (Plancherel identity).

 There exists a normal $\ast$-homomorphism called the comultiplication
 \[
 \Delta_{\widehat{G}}: L^\infty(\widehat{G}) \rightarrow L^\infty(\widehat{G}) \otimes L^\infty(\widehat{G}),
 \]
 that is determined by $\Delta_{\widehat{G}}(\lambda_G(g)) = \lambda_G(g) \otimes \lambda_G(g), g \in G$.

\subsection{Fourier multipliers} Let $1 < p < \infty$.  We call a function $m \in L^\infty(G)$ an $L^p$-Fourier multiplier if there is a bounded map $T_m: L^p(\widehat{G}) \rightarrow L^p(\widehat{G})$ that is determined by $\lambda_G(f) \mapsto \lambda_G(mf)$ for $f \in C_c(G) \ast C_c(G)$.

\subsection{Vector-valued $L^2$-spaces}\label{Sect=VecVal} Let $(X, \mu_X)$ be a regular Borel measure space and $\mathcal{X}$ be a Banach space. We write  $L^2(X; \mathcal{X})$ for the Banach space of locally Bochner integrable functions $f: X \rightarrow \mathcal{X}$ such that $\Vert f \Vert_{L^2(X; \mathcal{X})}  := \int_X \Vert f(x) \Vert_{\mathcal{X}}^2 d\mu(x)^{\frac{1}{2}} < \infty$.

\begin{remark}\label{Rmk=VectorFunctional}
Let $\varphi \in  L^2(X, \mu_X)$.  Then the map $(\varphi \otimes {\rm id}): L^2(X; \mathcal{X}) \rightarrow \mathcal{X}: f \mapsto \int_X f(x) \varphi(x) d\mu_X(x)$ is bounded with the same norm as $\varphi$ as easily follows from the Cauchy-Schwartz inequality.
\end{remark}

\section{Kernel estimates}\label{Sect=HeatKernel}
In \cite{AnkerJi} Anker and Ji showed that the fractional powers of the Casimir operator are represented by convolution kernels and they determined their asymptotic behavior. Consequently we can easily determine when these kernel are contained in  $L^p(G)$. These results from \cite{AnkerJi} follow from rather deep estimates on the Heat kernel associated with the Casimir operator that were conjectured in  \cite{AnkerOldPaper} and proved in several special cases before \cite{AnkerJi}.
We give some precise definitions borrowing some notation from \cite{AnkerJi}; but what really matters for us is \eqref{Eqn=Conv}. Define the Heat kernel, for $g \in G, t >0$,
\[
k_t(g) = \frac{C_{SF}}{ \vert W \vert} \int_{\fraka} \vert c(\lambda) \vert^{-2} e^{-t ( \Vert \lambda \Vert^2 + \Vert \rho \Vert^2 ) } \varphi_\lambda(g) d\lambda.
\]
Here $W$ is the Weyl group, $C_{SF} > 0$ a normalisation constant in the Spherical Fourier transform, $c$ the Harish-Chandra $c$-function and $\varphi_\lambda$ the spherical function indexed by $\lambda \in \fraka$ as in \cite{AnkerJi}.
Then  for $s >0$ we set the Bessel-Green-Riesz kernel using the $\Gamma$-function identity
\[
\kappa_s := \Gamma(s)^{-1} \int_0^\infty t^{s-1} k_t  dt.
\]
 These kernels are bi-$K$-invariant, as so is the spherical function $\varphi_\lambda$, and satisfy
\begin{equation}\label{Eqn=Conv}
e^{  - t \OmegaA} f = k_t \ast f, \qquad    \OmegaA^{-s} f = \kappa_s \ast f      \qquad f \in C_c(K \backslash G \slash K).
\end{equation}
 The following theorem is  essentially proved in \cite{AnkerJi}.

\begin{theorem}\label{Thm=AnkerJi}
For $0 < 2s <  \dim(G \slash K) $ and
\[
1 < q <  \frac{ \dim(  G/K )   }{ \dim(  G/K  )  - 2s},
\]
with moreover $q \leq 2$ we have that $\kappa_s$  is contained in $L^{q}( K \backslash G \slash K)$.
\end{theorem}
\begin{proof}
By the integral decomposition \eqref{Eqn=CartanIntegral} it suffices to show that $\kappa_s^{\fraka} := \kappa_s \circ \exp\vert_{\fraka^+}$ is contained in $L^{q}( \fraka^+, \delta(H) dH)$. We first consider the behavior of $\kappa_s^{\fraka}$ on  $B_{\geq 1}^+ := \{ H \in \fraka^+ \mid \Vert H \Vert \geq 1 \}$, i.e. away from the origin.
By \cite[Theorem 4.2.2]{AnkerJi} combined with \cite[Proposition  2.2.12 or Remark 4.2.3.(i)]{AnkerJi}   for $H \in B_{\geq 1}^+$,
\[
\begin{split}
\vert \kappa_s^{\fraka}( H ) \vert \preceq &  \Vert H \Vert^{s - \frac{ l + 1}{2} - \vert \SigmaPluss \vert}   e^{-  \Vert \rho \Vert   \Vert H \Vert - \langle \rho, H \rangle } \prod_{\alpha \in \SigmaPluss}  \left(   1 + \langle \alpha, H \rangle \right). \\
\end{split}
\]
 Recall the asymptotic behavior of the Haar measure $\delta(H) dH$ from \eqref{Eqn=CartanIntegral2}.
It follows that there exists a polynomial $P$ in $H$ such that
\begin{equation}\label{Eqn=AwayFromZero}
\int_{B_{\geq 1}^+} \vert \kappa_s^{\fraka}( H ) \vert^q  \delta(H) d H  \leq
\int_{B_{\geq 1}^+}
 e^{-q (  \Vert \rho \Vert   \Vert H \Vert + \langle \rho, H \rangle ) + 2 \langle \rho, H \rangle } P( \Vert H \Vert )
dH.
\end{equation}
As long as $q \leq 2$ we have for the exponent of the exponential function in the latter expression that
\[
-q (  \Vert \rho \Vert   \Vert H \Vert + \langle \rho, H \rangle ) + 2 \langle \rho, H \rangle
\leq (2 - 2q) \Vert \rho \Vert \Vert H \Vert.
\]
And therefore as long as $1 < q \leq 2$ we have that the integral \eqref{Eqn=AwayFromZero} is finite.

Next we consider the behavior of  $\kappa_s^{\fraka}$  on the region  $B_{\leq 1}^+ := \{ H \in \fraka^+ \mid \Vert H \Vert \leq 1 \}$ in order to have $\kappa_s^\fraka \in L^{q}( B_{\leq 1}^+ , \mu_A)$. By \cite[Remark 4.2.3.(iii)]{AnkerJi}  for $H \in B_{\leq 1}^+$ we have
\[
  \vert \kappa_s(\exp(H)) \vert \approx  \Vert H \Vert^{2s - \dim(G \slash K)}, \qquad 0 < 2 s <  \dim(G \slash K).
\]
It therefore follows, using again the expression of the Haar measure \eqref{Eqn=CartanIntegral}, that $\kappa_s$ is in $L^q(B_{\leq 1}, d\mu_{\fraka})$ if and only if
\[
 q (2s - \dim(G \slash K)) + \dim(N) > - \dim(A).
\]
This is equivalent to $q < \frac{\dim(G \slash K) }{ \dim(G \slash K) - 2s}$.
\end{proof}

Recall that $\OmegaA^{-s}$ with $s >0$ is a bounded operator on $L^2(K \backslash G \slash K)$. The next corollary concerns its extension to $L^p$-spaces of bi-$K$-invariant functions.

\begin{corollary}\label{Cor=CasimirBound}
For every $0 < 2s <  \dim(G \slash K)$ and  $\frac{\dim(G \slash K)}{ 2 s } < p < \infty$ with moreover $2 \leq p$, we have
\begin{equation}\label{Eqn=YoungPlus}
\Vert  \OmegaA^{-s}: L^{p}(K \backslash G \slash K ) \rightarrow  L^{\infty}(K \backslash G \slash K ) \Vert < \infty.
\end{equation}
\end{corollary}
\begin{proof}
By Theorem \ref{Thm=AnkerJi} the operator $\OmegaA^{-s}$ is a convolution operator with kernel $\kappa_s \in L^q(K \backslash G \slash K)$ where $\frac{1}{p} + \frac{1}{q} = 1$.
By Young's inequality for convolutions,
\[
\Vert \kappa_s \ast f \Vert_{ L^{\infty}( G  )} \leq
\Vert \kappa_s  \Vert_{ L^{q}(  G   ) }  \Vert f  \Vert_{ L^{p}(  G   ) },
\]
for any $f \in L^{p}(  G   )$. So the corollary follows from Theorem \ref{Thm=AnkerJi}.
\end{proof}

The following corollary is a special case of Theorem \ref{Thm=AnkerJi} in case  moreover $s < \frac{1}{2} \dim(G \slash K)$; this extra  condition is however not needed by Remark \ref{Cor=AssymptoticC}.

\begin{corollary}\label{Cor=L2Case}
For $s > \frac{1}{4} \dim(G \slash K)$ we have that $\kappa_s  \in L^{2}(  K \backslash G \slash K   )$.
\end{corollary}

\begin{remark}\label{Cor=AssymptoticC}
Corollary \ref{Cor=L2Case} does not necessarily require the results from \cite{AnkerJi} but  may also be derived more directly from the asymptotic behavior of the Harish-Chandra $c$-function as obtained in \cite[Eqn (3.44)]{Duistermaat}  by showing that the spherical Fourier transform of $\kappa_s$, given by $\widehat{\kappa_s}(\lambda) = (\Vert \rho \Vert^2 + \Vert \lambda\Vert^2)^{-s}, \lambda \in \fraka$, is in $L^2(\fraka, c^{-2}(\lambda) d\lambda)$.
 \end{remark}

\begin{remark}\label{Rmk=CowlingApplies}
As observed in   \cite[Proposition 3.1, Eqn. (3.12)]{GangolliActa} we have that the Heat kernel $k_t$ is the fundamental solution of the parabolic differential equation $-\OmegaA u  = \frac{\partial}{\partial t} u$ and as such $\Vert k_t  \Vert_{   L^1(K \backslash G \slash K) } = 1$ for $t >0$ \cite{Ito}; note that   \cite{GangolliActa} assumes $G$ is complex but here it is not relevant.  Therefore by Young's inequality
\[
 \Vert k_t \ast f \Vert_{L^p(G)} \leq \Vert f \Vert_{L^p(G)} \qquad f \in L^p(G), 1 \leq p \leq \infty,  t > 0.
\]
We already noted that $\OmegaA$ is a positive (unbounded) operator on $L^2(K \backslash G/K)$.
In particular we are in the setting of \cite{CowlingAnnals} and we will apply a result from \cite{CowlingAnnals} below.
\end{remark}

\section{$L^p$-estimates for radial multipliers}
In this section we prove Theorem \ref{Thm=Lp} which gives a first estimate for the norm of Fourier multipliers in terms of smoothness and regularity of the symbol with respect to the Casimir operator.

 We set
\[
P_K = \int_K \lambda_G(k) d\mu_K(k), \qquad P_K' = \int_K \lambda'_G(k) d\mu_K(k),
\]
as the orthogonal projection of $L^2(G)$ onto $L^2(K \backslash G)$ and $L^2(G \slash K)$, respectively.
 We define
\[
L^\infty(\widehat{ K \backslash G \slash K}) := \{  \lambda_G(f)   \mid f \in L^1(K \backslash G \slash K)  \}'',
\]
which is a von Neumann algebra acting on $L^2(G)$.
As $(G,K)$ forms a Gelfand pair $L^\infty(\widehat{ K \backslash G \slash K})$ is commutative.
Set further
\[
\mathbb{E}_K(x) = P_K x P_K, \qquad  x \in L^\infty(\widehat{G}),
\]
which we view as an operator on $L^2(G)$. Note that for $ f \in L^1(G)$ we have
\[
\mathbb{E}_K(  \lambda_G(f) ) = \lambda_G( \widetilde{f}), \quad \textrm{ where } \quad  \widetilde{f}(g) = \int_K \int_K  f(k_1 g k_2  ) dk_1 dk_2.
\]
and hence we see that $\mathbb{E}_K$ is a normal $\tau_{\widehat{G}}$-preserving conditional expectation of $L^\infty(\widehat{G})$ onto $L^\infty(\widehat{ K \backslash G \slash K})$. We denote $L^p(\widehat{ K \backslash G \slash K})$ for the $L^p$-space constructed from $L^\infty(\widehat{ K \backslash G \slash K})$ with trace $\tau_{\widehat{G}}$. Then $L^p(\widehat{ K \backslash G \slash K})$ is a closed subspace of $L^p(\widehat{G})$. All von Neumann algebras we defined above act on $L^2(G)$ whereas the Casmir operator acts on $L^2(K \backslash G \slash K)$. Through the following lemma we show that it naturally acts on $L^2(G)$ as well.

\begin{lemma}\label{Lem=PolarVNA}
The map $\pi: \lambda_G(f) \mapsto \lambda_G(f) P_K'$ extends to an isomorphism of von Neumann algebras $M_1 \rightarrow M_2$ where
\[
M_1 := L^\infty(\widehat{ K \backslash G \slash K}), \qquad \textrm{ and } \qquad  M_2 =   M_1 P_K'.
\]
\end{lemma}
\begin{proof}
As $P_K'$ commutes with $M_1$ we have that $\pi$ is a normal $\ast$-homomorphism and it remains to show that it is injective. Set $A^+ =    \exp(\fraka^+)$ and consider
\[
\varphi: K \times  A^+   \times K \rightarrow G: (k_1, a, k_2) \mapsto k_1 a k_2.
\]
The complement of the range of $\varphi$ in $G$ has measure 0 \cite[Section xix]{JorgensonLang}; note that the map is generally not injective as $M$ (the centralizer of $A$ in $K$, see Section \ref{Sect=Prelim}) is usually non-trivial.
Then the pullback map
\[
U := \varphi^\ast: L^2(G) \rightarrow L^2(K \times   A^+  \times K)  \simeq L^2(K \times  A^+    ) \otimes L^2(K ).
\]
is isometric by \eqref{Eqn=CartanIntegral} if we equip $K$ with its Haar measure and $A^+$ with the pullback along $\exp: \fraka^+ \rightarrow A^+$ of the measure $\vert K \slash M \vert   \delta(H) dH$ on $\fraka^+$. Let $P_U = U U^\ast$ be the range projection of $U$.  Set 
\[
\widetilde{P}_K = \int_K \lambda_K(k) d\mu_K(k) =  \int_K \lambda'_K(k) d\mu_K(k),
\]
which is the projection of $L^2(K)$ onto the constant functions. 

  We have  $(1 \otimes \lambda_K'(k))  U  =  U \lambda_G'(k)$ for all $k \in K$. It follows that
\[
U P_K' = (1 \otimes \widetilde{P}_K) U \quad \textrm{  and so } \quad (1 \otimes \widetilde{P}_K) P_U =   P_U (1 \otimes \widetilde{P}_K).
\]
 Take  $x \in M_1$. As $x$ commutes with every $\lambda'_G(k), k \in K$ we find that
\begin{equation}\label{Eqn=Inclusion}
U  x U^\ast    \in B(  L^2(K \times    A^+   )   ) \otimes L^\infty(\widehat{K}).
\end{equation}

If for $x \in M_1$ we have that  $x P_K' = 0$ then it follows that
\[
0 = Ux P_K'  = U x U^\ast U P_K'  = U x  U^\ast  (1 \otimes \widetilde{P}_K) U.
\]
But, multiplying with $U^\ast$ from the right,  this means that
\[
0 = U x  U^\ast (1 \otimes \widetilde{P}_K) P_U = U x  U^\ast P_U (1 \otimes \widetilde{P}_K)  = U x  U^\ast (1 \otimes \widetilde{P}_K).
\]
 So the range of $1 \otimes \widetilde{P}_K$ given by  $L^2(K \times   A^+  ) \otimes \mathbb{C} 1$  is in the kernel of $Ux U^\ast$ and hence, by \eqref{Eqn=Inclusion}, we have   $Ux U^\ast = 0$. This yields that $x = U^\ast U x U^\ast U = 0$.    This shows that $\pi$ is injective.

\end{proof}

 It follows in particular that $\pi$ in Lemma \ref{Lem=PolarVNA} restricts to an isomorphism $\pi'$ of $L^\infty(K \backslash G \slash K)$ onto its image. As this image consists of bi-$K$-invariant functions we see that $\pi'(x) =  x P_K P_K'$   for $x \in L^\infty(K \backslash G \slash K)$. Thus the restriction map $\pi_K: x \mapsto x\vert_{L^2(K \backslash G \slash K) }$ yields an isomorphism of $L^\infty(\widehat{K \backslash G \slash K})$ onto its image. Now the Casimir operator $\OmegaA$ is affiliated with $\pi_K(L^\infty(\widehat{K \backslash G \slash K}))$; indeed all its finite spectral projections are in  this von Neumann algebra, see \cite[Eqn. (4.2.1)]{AnkerJi},  \cite{HelgasonSpherical}. There we may see $\OmegaA$ as an operator affiliated with $L^\infty(\widehat{K \backslash G \slash K})$ and thus acting on $L^2(G)$ as well. We note that formally $\OmegaA$ is equal to the  Casimir operator {\it restricted} to the bi-$K$-invariant functions.

The following is now a consequence of Corollary \ref{Cor=L2Case}, the spectral theorem for the Casimir operator \cite[Eqn. (4.2.1)]{AnkerJi} and the Plancherel identity/unitarity of the spherical Fourier transform.

\begin{corollary}\label{Cor=L2CaseCasimir}
For $s > \frac{1}{4} \dim(G \slash K)$ we have that $\OmegaA^{-s}  \in L^{2}(  \widehat{ K \backslash G \slash K }  )$.
\end{corollary}

  Now for $s>0$ consider the following completely positive map,
 \begin{equation}\label{Eqn=TMap}
 T_{s} =   (  \OmegaA^{-s} \otimes 1) ( \mathbb{E}_K \otimes \id) \circ \Delta_{\widehat{G}}: L^\infty(\widehat{G}) \rightarrow L^\infty(\widehat{K \backslash G \slash K}) \otimes  L^\infty(\widehat{G}).
 \end{equation}

 We will need vector valued $L^p$-spaces in case the measure space is the underlying measure space of the commutative von Neumann algebra $L^\infty(\widehat{K \backslash G \slash K})$. The following definition is rather explicit for $p =1$ and $p=2$ and all our other results will follow from complex interpolation between these cases. It agrees with the one in Section \ref{Sect=VecVal} if one identifies $L^\infty(\widehat{K \backslash G \slash K})$ with $L^\infty(X, \mu_X)$ for suitable $X$; in fact one may take $(X, \mu_X) = (\fraka, c(\lambda)^{-2} d\lambda)$ where $c$ is the Harish-Chandra $c$-function.

\begin{definition}
 We define the  $L^p(\widehat{G}   )$-valued $L^2$-space  $L^2( \widehat{K \backslash G \slash K};  L^p(\widehat{G}   ))$ as follows.  For $p=1,2$ it is the completion of
 \[
 L^\infty(\widehat{K \backslash G \slash K}) \otimes  L^\infty(\widehat{G}) \cap L^1(\widehat{K \backslash G \slash K}) \otimes  L^1(\widehat{G})
 \]
  with respect to the respective norms:
 \begin{equation}\label{Eqn=MixedNorm}
 \begin{split}
    \Vert y \Vert_{   L^2( \widehat{K \backslash G \slash K};  L^1(\widehat{G}   ))   } = &
      \tau_{\widehat{G}} ( \vert (\id \otimes \tau_{\widehat{G}})(\vert y\vert)  \vert^2)^{\frac{1}{2}}. \\
        \Vert y \Vert_{   L^2( \widehat{K \backslash G \slash K};  L^2(\widehat{G}   ))   } = &
      (\tau_{\widehat{G}}   \otimes \tau_{\widehat{G}}  ) (y^\ast y)^{\frac{1}{2}} = \Vert y \Vert_{   L^2( \widehat{K \backslash G \slash K}) \otimes L^2(\widehat{G}) }.
 \end{split}
 \end{equation}
  Further, through complex interpolation we may isometrically identify
  \[
  L^2( \widehat{K \backslash G \slash K};  L^p(\widehat{G}   )) \simeq [L^2( \widehat{K \backslash G \slash K};  L^1(\widehat{G}   )), L^2( \widehat{K \backslash G \slash K};  L^2(\widehat{G}   ))]_{\theta},
  \]
  where $\theta$ is such that $\frac{1}{p} = \frac{1-\theta}{2} + \frac{\theta}{1}$.
\end{definition}

  Typically we want $S_G$ in the next lemma to be very close to $\frac{1}{4} \dim(G \slash K)$ to get sharp estimates.

\begin{lemma}\label{Lem=TSEstimate}
Let $S_G > \frac{1}{4} \dim(G \slash K)$. Then, for $1 \leq p \leq 2$,
 \begin{equation}\label{Eqn=TS}
 \begin{split}
\Vert  T_{S_G}: L^p(\widehat{G}) \rightarrow  L^2( \widehat{K \backslash G \slash K};    L^p(\widehat{G}   )) \Vert \leq & 2 \Vert \OmegaA^{-S_G} \Vert_{ L^2(\widehat{K \backslash G \slash K}) } < \infty.
 \end{split}
 \end{equation}
 \end{lemma}
 \begin{proof}
 That the right hand side of \eqref{Eqn=TS} is finite follows from Corollary \ref{Cor=L2CaseCasimir}.
 We prove this for $p=1$ and $p=2$ so that the lemma follows from complex interpolation.
For $p=2$ we find for $x \in L^2(\widehat{G}) \cap L^\infty(\widehat{G})$ by the Kadison-Schwarz inequality,
 \[
 \begin{split}
    \Vert T_{S_G}(x) \Vert_{   L^2( \widehat{K \backslash G \slash K};  L^2(\widehat{G}   ))   }  ^2 = & (\tau_{\widehat{G}} \otimes \tau_{\widehat{G}})(T_{S_G}(x)^\ast T_{S_G}(x)) \\
= &  (\tau_{\widehat{G}} \otimes \tau_{\widehat{G}}) \left( (\OmegaA^{-S_G}  \otimes 1)  (\mathbb{E}_K  \otimes \id) (\Delta_{\widehat{G}}(x))^\ast   (\mathbb{E}_K  \otimes \id) (\Delta_{\widehat{G}}(x))       (\OmegaA^{-S_G}  \otimes 1)   \right) \\
 \leq & (\tau_{\widehat{G}} \otimes \tau_{\widehat{G}}) \left( (\OmegaA^{-S_G}  \otimes 1)  (\mathbb{E}_K  \otimes \id) (\Delta_{\widehat{G}}(x^\ast x))        (\OmegaA^{-S_G}  \otimes 1)   \right).
 \end{split}
 \]
  Then by left invariance of the Plancherel trace $\tau_{\widehat{G}}$ on $L^\infty(\widehat{G})$ (see \cite[Definition 1.1]{KustermansVaes}),
  \[
     \Vert T_{S_G}(x) \Vert_{   L^2( \widehat{K \backslash G \slash K};  L^2(\widehat{G}   ))   } ^2 \leq \tau_{\widehat{G}}( \OmegaA^{-2S_G}  \mathbb{E}_K (1)    ) \tau_{\widehat{G}}(x^\ast x) =  \tau_{\widehat{G}}( \OmegaA^{-2S_G }   ) \tau_{\widehat{G}}(x^\ast x) =
  \tau_{\widehat{G}}( \OmegaA^{-2S_G}     ) \Vert x \Vert_{L^2(\widehat{G} )}^2.
  \]
This proves the $L^2$-estimate.

  For the $p=1$ estimate,  take  $x \in L^1(\widehat{G})$  positive so that  $T_{S_G}(x)$ is positive.
   By \eqref{Eqn=MixedNorm} and using left-invariance of $\tau_{\widehat{G}}$ twice,
  \begin{equation}\label{Eqn=L12Est}
  \begin{split}
  \Vert T_{S_G}(x) \Vert_{   L^2( \widehat{K \backslash G \slash K};  L^1(\widehat{G}   ))   }^2 = &
  \tau_{\widehat{G}}\left( \OmegaA^{-S_G} \vert (\mathbb{E}_K \otimes \tau_{\widehat{G}}) \Delta_{\widehat{G}}(x)   \vert^2 \OmegaA^{-S_G} \right)\\
  \leq &
  \tau_{\widehat{G}}\left( \OmegaA^{-S_G}
  \mathbb{E}_K\left(
  (\id \otimes \tau_{\widehat{G}} )   ( \Delta_{\widehat{G}}(x^\ast) )
  (\id \otimes \tau_{\widehat{G}} )   ( \Delta_{\widehat{G}}(x) )
  \right)  \OmegaA^{-S_G} \right)   \\
  = & \tau_{\widehat{G}}\left( \OmegaA^{-S_G} (\mathbb{E}_K \otimes \tau_{\widehat{G}}  )( \Delta(x^\ast)    )    \OmegaA^{-S_G} \right)
   \tau_{\widehat{G}}(x) \\
   = &  \tau_{\widehat{G}}( \OmegaA^{-2S_G})
   \vert \tau_{\widehat{G}}(x)\vert^2 =  \tau_{\widehat{G}}( \OmegaA^{-2S_G}) \Vert x \Vert_1^2.
  \end{split}
  \end{equation}
 For general (non-positive) self-adjoint $x \in L^1(\widehat{G})$ we use the fact that it can be written as $x = x_1 - x_2$ with $x_i$ positive and  $\Vert x_1 \Vert_1 + \Vert x_2 \Vert_1  = \Vert x \Vert_1$ to conclude the same estimate \eqref{Eqn=L12Est}. For general $x \in L^1(\widehat{G})$ we use the fact that it can be written as $x = x_1 + ix_2$ with $x_i$ self-adjoint and $\Vert x_i \Vert \leq \Vert x \Vert$ and conclude
  \[
  \Vert T_{S_G}(x) \Vert_{   L^2( \widehat{K \backslash G \slash K};  L^1(\widehat{G}   ))   } \leq 2  \tau_{\widehat{G}}( \OmegaA^{-2S_G})^{\frac{1}{2}} \Vert x \Vert_1.
  \]
  This concludes the proof.
  \end{proof}

\begin{lemma} \label{Lem=Differential}
Let $s > 0$. For $m \in L^1(K \backslash G \slash K) \cap L^2(K \backslash G \slash K)$ in the domain of $\OmegaA^{s}$ such that $\OmegaA^{s}(m) \in L^1(K \backslash G \slash K)$ we have
\[
  \OmegaA^{s} \lambda_G(m) =  \lambda_G( \OmegaA^s (m)),
\]
and in particular for every $\xi \in L^2(G)$ we have that $\lambda_G(m) \xi$ is in the domain of $\OmegaA^{s}$.
\end{lemma}
 \begin{proof}
Take $f \in C_c(G)$. As  $\OmegaA^{s}$ is  affiliated with $L^\infty(\widehat{K \backslash G \slash K})$  it commutes with right convolutions and so we have
\[
 \OmegaA^{s} \lambda_G(m) f =  \OmegaA^{s} (m \ast f) = \OmegaA^{s}(m) \ast f = \lambda_G( \OmegaA^s (m) ) f.
\]
Now $\OmegaA^{s} \lambda(m)$ is a closed operator; this follows from the general fact that if $d$ is closed and $x$ is bounded then $dx$ with domain $\{ \xi \mid x \xi \in {\rm Dom}(d)\}$ is closed. The assumption that $\OmegaA^s (m)  \in L^1(K \backslash G \slash K)$ assures that
$\lambda_G( \OmegaA^s (m))$ is bounded. Further, as we have shown that $\lambda_G( \OmegaA^s (m))$  equals $\OmegaA^{s} \lambda_G(m)$  on the domain  $C_c( G)$ it follows that   $\OmegaA^{s} \lambda_G(m) =  \lambda_G( \OmegaA^s (m))$ as operators on $L^2(  G)$.
 \end{proof}



 \begin{theorem}\label{Thm=Lp}
 Let $1 \leq p \leq 2$.  Let   $m \in L^2(K \backslash G \slash K)   \cap L^\infty(K \backslash G \slash K)$ with $m \in \Dom(\OmegaA^{S_G})$. Then, for $S_G > \frac{1}{4} \dim(G \slash K)$,
 \[
\Vert  T_m: L^p(\widehat{G}) \rightarrow  L^p(\widehat{G}) \Vert \leq 2  \Vert \OmegaA^{S_G}(  m) \Vert_{L^2(G)}   \Vert \OmegaA^{-S_G} \Vert_{L^2(\widehat{K \backslash G \slash K})}.
 \]
 \end{theorem}
 \begin{proof}
 Let $(U_i)_i$ be shrinking bi-$K$-invariant neighbourhoods of the identity of $G$ such that $\cap_i U_i = K$ and $U_i^{-1} = U_i$.
  Let $m_i =    m \ast I_i$ where $I_i = \vert U_i \vert^{-1} 1_{U_i}$ is an $L^1(G)$-normalisation of the indicator function on $U_i$. As
  \begin{equation}\label{Eqn=FourierAlgebra}
  \OmegaA^{S_G}(m_i) =
\OmegaA^{S_G}(m \ast I_i) = \OmegaA^{S_G}(m)  \ast I_i,
\end{equation}
  it follows that $m_i$ satisfies the same assumptions as made on $m$ in the statement of the proposition. Suppose that we have proved the proposition for $m_i$ then by taking limits in $i$ it also follows for $m$. Now $m_i$ has the additional property that it is contained in the   Fourier algebra $A(G)$ of $G$ (see Eymard \cite{Eymard}) meaning that
\begin{equation}\label{Eqn=VarphiM}
\varphi_{m_i}: \lambda_G(f) \mapsto \int_G m_i(g) f(g) d\mu_G(g),
\end{equation}
extends to a normal bounded functional on $L^\infty(\widehat{G})$. The equation \eqref{Eqn=FourierAlgebra} similarly shows that $\OmegaA^{S_G}(m_i) \in A(G)$ and hence $\varphi_{\OmegaA^{S_G}(m_i)}$ may also be defined as a normal bounded functional on $L^\infty(\widehat{G})$ by replacing $m_i$ by $\OmegaA^{S_G}(m_i)$ in \eqref{Eqn=VarphiM}.
We first derive a number of properties for our setup that shall be used in the core of our proof.

\vspace{0.3cm}

\noindent (1) Note that as $m_i$ is bi-$K$-invariant,  for $f \in L^1(G)$ we have, using the left and right invariance of the Haar measure,
 \[
 \begin{split}
 \varphi_{m_i}(\mathbb{E}_K( \lambda_G(f)) ) = &  \int_K \int_K \int_G m_i(g) f(k_1 g k_2) d\mu_G(g) d\mu_K(k_1) d \mu_K(k_2)\\ = &
 \int_G m_i( g  ) f(g) d\mu_G(g)
  =
 \varphi_{m_i}( \lambda_G(f) ).
  \end{split}
 \]
So $\varphi_{m_i} \circ \mathbb{E}_K = \varphi_{m_i}$.

\noindent (2) We have by Lemma \ref{Lem=Differential} for $f \in L^1(K \backslash G \slash K) \cap L^2(K \backslash G \slash K)$ in the domain of $\OmegaA^{S_G}$ such that $\OmegaA^{S_G}(f) \in L^1(K \backslash G \slash K)$ that
 \[
 \begin{split}
 \varphi_{m_i}( \OmegaA^{S_G} \lambda_G(f)  ) = &
 \varphi_{m_i}( \lambda_G(  \OmegaA^{S_G} f)  ) =
 \int_G m_i( g  ) (\OmegaA^{S_G}  f)(g) d\mu_G(g) \\
  = &
 \int_G  (\OmegaA^{S_G} m_i)( g  )  f(g) d\mu_G(g)   =
 \varphi_{\OmegaA^{S_G} (m_i)} ( \lambda_G(f)  ).
  \end{split}
 \]
In particular, $\lambda_G(f) \mapsto \varphi_{m_i}( \OmegaA^{S_G} \lambda_G(f)  )$ extends to a normal map on $L^\infty(\widehat{K \slash G \backslash K})$.

\noindent (3) As $m_i \in L^2(K \backslash G \slash K)$ it follows directly from \eqref{Eqn=FourierAlgebra} that
\begin{equation}\label{Eqn=L2VecEst}
 \Vert \varphi_{   \OmegaA^{S_G} (m_i)  } \Vert_{L^2(   \widehat{K \backslash G \slash K}  )^\ast}  \leq \Vert \OmegaA^{S_G} (m_i) \Vert_{  L^2(G) }.
 \end{equation}

\vspace{0.3cm}


We now come to the main part of the proof. Now take $f \in C_c(G)^{\ast 2}$. Then,
 \begin{equation}\label{Eqn=TmComp}
 \begin{split}
 T_{m_i}(\lambda_G(f)) = &  \int_G m_i(g) f(g) \lambda_G(g)  d\mu_G(g)
 =    (\varphi_{m_i}  \otimes \id) \left(  \int_G f(g) \lambda_G(g) \otimes \lambda_G(g)  d\mu_G(g) \right) \\
 = & (  \varphi_{m_i}   \otimes \id) \circ \Delta_{\widehat{G}}(\lambda_G(f))
 =  (  \varphi_{m_i} \circ \mathbb{E}_K   \otimes \id) \circ \Delta_{\widehat{G}}(\lambda_G(f)).
  \end{split}
 \end{equation}
   Note   that   $(\OmegaA^{S_G} \otimes 1) (\OmegaA^{-S_G} \otimes 1)$ equals the unit of $L^\infty(K \backslash G \slash K)$, in particular with equality of domains.
 Therefore we get, and this is the most crucial equality in this paper,
 \begin{equation}\label{Eqn=TmComp2}
 \begin{split}
T_{m_i} =  &   (  \varphi_{m_i}  \otimes \id) \left( (\OmegaA^{S_G} \otimes 1) (\OmegaA^{-S_G} \otimes 1) (\mathbb{E}_K \otimes \id)  \Delta_{\widehat{G}}  \right) \\
= & (   \varphi_{\OmegaA^{S_G} (m_i)}  \otimes \id) \left(   (\OmegaA^{-S_G} \otimes 1) (\mathbb{E}_K \otimes \id)  \Delta_{\widehat{G}}  \right)
=  (  \varphi_{\OmegaA^{S_G} (m_i)}  \otimes \id) \circ T_{S_G}.
 \end{split}
 \end{equation}
It follows by this equation, Remark \ref{Rmk=VectorFunctional}, \eqref{Eqn=L2VecEst} and Lemma \ref{Lem=TSEstimate} that
\[
\begin{split}
\Vert  T_{m_i}: L^p(\widehat{G}) \rightarrow  L^p(\widehat{G}) \Vert \leq & \Vert \varphi_{  \OmegaA^{S_G} (m_i)  } \Vert_{L^2(\widehat{K \backslash G \slash K})^\ast }
 \Vert T_{S_G}: L^p(\widehat{G}) \rightarrow  L^2(\widehat{K \backslash G \slash K}  ;   L^p(\widehat{G}) )    \Vert \\
 \leq & 2 \Vert  \OmegaA^{S_G} (m_i) \Vert_{L^2(G)} \Vert \OmegaA^{-S_G} \Vert_{ L^2(\widehat{K \backslash G \slash K}) }.
\end{split}
\]
The theorem now follows by taking limits in $i$ as justified in the beginning of the proof.
 \end{proof}

\section{Interpolation between $L^p$ and $L^2$ and conclusion of the main theorem}\label{Sect=Interpolation}

The result in this section should be seen as a complex interpolation result between the estimate from Theorem \ref{Thm=Lp} and the bound obtained in Corollary \ref{Cor=CasimirBound} that followed from the analysis by Anker and Ji \cite{AnkerJi}.  Similar results can be found in the literature (see \cite{Triebel}, \cite{GrafakosIMRN}) but we have not found a theorem that was directly applicable and therefore we provide a self-contained proof. 
 We use the following variation of the three lines lemma which can be found in \cite{GrafakosClassical} or \cite{Hirschman}. We define the usual strip
\[
\mathcal{S} = \{ z \in \mathbb{C} \mid 0 \leq \Im(z) \leq 1 \}.
\]

\begin{lemma}\label{Lem=ThreeLines}
Let $F: \mathcal{S} \rightarrow \mathbb{C}$ be continuous and analytic on the interior of $\mathcal{S}$. Assume that  for every $0 \leq \beta \leq 1$ there exists a function $A_\beta: \mathbb{R} \rightarrow \mathbb{R}_{> 0}$ and scalars $A >0$ and $0<a<\pi$ such that such that for all $t \in \mathbb{R}$ we have
\[
F(\beta + it) \leq A_\tau(t) \leq e^{   A e^{a \vert t \vert} }.
\]
Then for $0 < \beta < 1$ we have $\vert F(\beta) \vert \leq e^{D_\beta}$ where
\[
D_\beta = \frac{\sin(\pi \theta)}{2} \int_{-\infty}^\infty \frac{\log(\vert A_0(t) \vert) }{ \cosh(\pi t) - \cos(\pi \beta) } + \frac{\log \vert A_1(t) \vert}{ \cosh(\pi t) + \cos(\pi \beta) } dt.
\]
\end{lemma}

To apply this lemma it is crucial to realize that
\[
\frac{\sin(\pi \theta)}{2} \int_{-\infty}^\infty \frac{ dt  }{ \cosh(\pi t) - \cos(\pi \beta) } = 1 - \theta, \qquad
 \frac{\sin(\pi \theta)}{2} \int_{-\infty}^\infty  \frac{ dt }{ \cosh(\pi t) + \cos(\pi \beta) }.
\]

The following is now the main theorem of this paper.

\begin{theorem}\label{Thm=Main}
Let $ \frac{1}{4} \dim(G/K) < S_G <  \frac{1}{2} \dim(G/K)$.
Let $p \in (1,\infty)$.  Let $s \in (0,  S_G]$    be such that
\[
\left| \frac{1}{p} - \frac{1}{2} \right| < \frac{s}{2 S_G}.
\]
Then, there exists a constant $C_{G, s,p} >0 $ only depending on the group $G$ and the exponents $s$ and $p$, such that for every $m \in L^2(K \backslash G \slash K) \cap L^\infty(K \backslash G \slash K)$ with $m \in \Dom(\OmegaA^s)$ and $\OmegaA^s(m) \in L^{2S_G/s}(G)$ we have,
\begin{equation}\label{Eqn=MainTheorem}
\Vert T_m: L^p(\widehat{G}) \rightarrow  L^p(\widehat{G}) \Vert \leq C_{G, s,p} \Vert \OmegaA^s(m) \Vert_{L^{2S_G/s}(G)}.
\end{equation}
\end{theorem}
\begin{proof}
For completeness we mention that for $p=2$ this result is Theorem \ref{Thm=Lp} in combination with Corollary \ref{Cor=L2Case}; of course for $p=2$ the estimate in the theorem is very crude as $\Vert m \Vert_{L^\infty(G)}$ is the norm of $T_m$ in  \eqref{Eqn=MainTheorem}.
Assume $p \not = 2$. By duality it suffices to treat the case $p \in (1,2)$.
Take $\alpha \in (0,1)$ such that
\[
\frac{1}{p} - \frac{1}{2} < \frac{\alpha  s}{2 S_G}.
\]
Set $p_1 = \frac{2}{\alpha +1}$ and $s_1 = S_G$.
Set $p_0 = 2$. Set
\begin{equation}\label{Eqn=SGEstimate}
\theta := \left(\frac{1}{p} - \frac{1}{2} \right) \frac{2}{\alpha} < \frac{ s}{S_G} \leq 1.
\end{equation}
Hence $\theta \in [0,1]$ and further $\frac{1}{p} = \frac{1-\theta}{2} + \frac{\theta}{p_1}$; in particular $p_1 < p < 2$. In \eqref{Eqn=SGEstimate} we have already noted that $\theta S_G < s \leq S_G$. Therefore we may pick $s_0 \in (0, s)$ such that
\[
s = (1-\theta) s_0 + \theta S_G = (1-\theta) s_0 + \theta s_1.
\]
Note that   
\begin{equation}\label{Eqn=FixEstimate}
s_0 < s \leq S_G < \frac{1}{2} \dim(G \slash K) \textrm{ and } s_1 =S_G < \frac{1}{2} \dim(G \slash K).
\end{equation}
The idea of the rest of the proof is to interpolate between $(p_0,s_0)$ and $(p_1, s_1)$ by means of Lemma \ref{Lem=ThreeLines}.

\vspace{0.3cm}

\noindent {\it Step 1: Defining the function $F$.}
Recall that $s$ was fixed in the statement of the theorem and set,
\[
m_s = \OmegaA^s(m) \in L^{2S_G/s}(K \backslash G \slash K).
\]
For $z \in \mathcal{S}$ set
\[
M_z = \OmegaA^{-(1-z) s_0 - z s_1 }( m_s \vert m_s \vert^{\frac{-s +  (1-z) s_0 + z s_1 }{s} }  );
\]
we need to argue how the application of  $\OmegaA^{-(1-z) s_0 - z s_1 }$ is interpreted, and we shall do that in Step 1a where we show that it is a bounded operator from
$L^{q_\beta}$ to $L^\infty$ (notation below) and at the same time we show that $M_z$ is a function in $L^\infty(G)$.  At this point we observe already that
\begin{equation}\label{Eqn=Mtheta}
M_\theta =  \OmegaA^{- s }( m_s  ) = m.
\end{equation}
 Let $p', p_0'$ and $p_1'$ be the conjugate exponents of respectively $p, p_0$ and $p_1$.
Take $f_1 \in C_c(G)^{\ast 2}$ and set $f = f_1^\ast \ast f_1$.  Similarly, take $g_1 \in C_c(G)$ and set $g = g_1^\ast \ast g_1 \ast \ldots \ast g_1^\ast \ast g_1$ with $k\in \mathbb{N}$ occurrences of $g_1^\ast \ast g_1$ where $k \geq \frac{p'}{p_1'}$. Set $a = \lambda_G(f), b \in \lambda_G(g)$ which are positive and
 contained in $L^\infty(\widehat{G}) \cap L^1(\widehat{G})$. Our assumptions moreover imply that $a^z \in L^\infty(\widehat{G}) \cap L^2(\widehat{G})$ as long as  $\Re(z) \geq \frac{1}{2}$. Further, $b^z \in L^\infty(\widehat{G}) \cap L^2(\widehat{G})$ as long as  $\Re(z) \geq \frac{1}{2k}$. So surely all complex powers of $a$ and $b$ in the expression \eqref{Eqn=Fz} below are contained in $L^\infty(\widehat{G}) \cap L^2(\widehat{G})$. Further, the application of $T_{M_z}$  in \eqref{Eqn=Fz} is justified as it is a bounded map on  $L^2(\widehat{G})$.
So we define,
\begin{equation}\label{Eqn=Fz}
 F(z) =  \tau_{\widehat{G}}( T_{M_z}(  a^{(1-z) \frac{p}{p_0} + z \frac{p}{p_1} } )    b^{  (1-z) \frac{p'}{p_0'} + z \frac{p'}{p_1'} }     ).
\end{equation}
Then $F$ is continuous on $\mathcal{S}$ and analytic on the interior of $\mathcal{S}$.  We now require 3 estimates on $F$.

\vspace{0.3cm}

\noindent {\it Step 1a: Estimating $F$ on the strip $\mathcal{S}$.}
 For any $z \in \mathcal{S}$ we have,
\[
\vert F(z) \vert \leq \Vert M_z \Vert_{L^\infty(G)} \Vert  a^{(1-z) \frac{p}{p_0} + z \frac{p}{p_1} }  \Vert_{L^2(\widehat{G})}  \Vert  b^{  (1-z) \frac{p'}{p_0'} + z \frac{p'}{p_1'} }  \Vert_{L^2(\widehat{G})}.
\]
Here the terms
\[
\Vert  a^{(1-z) \frac{p}{p_0} + z \frac{p}{p_1} }  \Vert_{L^2(\widehat{G})}, \quad \textrm{ and } \quad \Vert  b^{  (1-z) \frac{p'}{p_0'} + z \frac{p'}{p_1'} }  \Vert_{L^2(\widehat{G})},
\]
 are uniformly bounded in $z \in \mathcal{S}$. Now write $z = \beta + it, \beta \in [0,1], t \in \mathbb{R}$. Set $s_\beta = (1-\beta)s_0 + \beta s_1$ and then  $q_\beta = 2S_G/s_\beta$. Set $q = 2S_G/s$. So $s = s_\theta$ and $q = q_\theta$. By \eqref{Eqn=FixEstimate} we have $s_\beta < \frac{1}{2} \dim(G \slash K)$. 
   We estimate,
\begin{equation}\label{Eqn=StripEstimate}
\begin{split}
\Vert M_z \Vert_{L^\infty(G)}  =  & \Vert \OmegaA^{-s_\beta +it (s_0 - s_1)}( \OmegaA^s(m)  \vert\OmegaA^s(m)\vert^{\frac{-s +  (1-z) s_0 + z s_1 }{s} }  ) \Vert_{L^\infty(G)}  \\
\leq & \Vert \OmegaA^{- s_\beta}: L^{q_\beta}(K \backslash G \slash K) \rightarrow L^\infty(K \backslash G \slash K) \Vert \\
& \:  \times \:  \Vert \OmegaA^{ it (s_0 - s_1) }: L^{q_\beta}(K \backslash G \slash K) \rightarrow L^{q_\beta}(K \backslash G \slash K) \Vert \Vert  \vert\OmegaA^s(m)\vert^{ \frac{s_\beta}{s } }  \Vert_{ L^{q_\beta}(K \backslash G \slash K) }.
\end{split}
\end{equation}
By Corollary \ref{Cor=CasimirBound} and using that $q_\beta = 2S_G/s_\beta > \dim(G \slash K)/2 s_\beta$, we have
\begin{equation}\label{Eqn=ConvEstimate}
C_{\beta}  := \Vert \OmegaA^{- s_\beta}: L^{q_\beta}(K \backslash G \slash K) \rightarrow L^\infty(K \backslash G \slash K) \Vert < \infty.
\end{equation}
By Remark \ref{Rmk=CowlingApplies} the Heat semi-group is a contractive semi-group with positive generator and hence falls within the setting of \cite{CowlingAnnals}. By \cite[Corollary 1]{CowlingAnnals} there exists a constant $C_\beta' > 0$ only depending on $\beta$ such that
\begin{equation}\label{Eqn=CowlingEstimate}
 \Vert \OmegaA^{ it (s_0 - s_1) }: L^{q_\beta}(K \backslash G \slash K) \rightarrow L^{q_\beta}(K \backslash G \slash K) \Vert
 \leq C_\beta' (1 + \vert t \vert^3 \log^2(\vert t \vert) )^{\vert \frac{1}{q_\beta} - \frac{1}{2} \vert}.
\end{equation}
Finally note that
\begin{equation}\label{Eqn=SwitchingSpace}
\Vert  \vert\OmegaA^s(m)\vert^{ \frac{s_\beta}{s } }  \Vert_{ L^{q_\beta}(K \backslash G \slash K) } =
\Vert   \OmegaA^s(m)  \Vert_{ L^{q}(K \backslash G \slash K) }^{\frac{q}{q_\beta}}.
\end{equation}
Combining \eqref{Eqn=StripEstimate} with \eqref{Eqn=ConvEstimate},  \eqref{Eqn=CowlingEstimate} and \eqref{Eqn=SwitchingSpace} yields
\begin{equation}\label{Eqn=MasterEstimate}
\begin{split}
\vert F(z) \vert \leq   &   C_\beta C_\beta' (1 + \vert t \vert^3 \log^2(\vert t \vert) )^{\vert \frac{1}{q_\beta} - \frac{1}{2} \vert}  \Vert   \OmegaA^s(m)  \Vert_{ L^{q}(K \backslash G \slash K) }^{\frac{q}{q_\beta}} \\
& \: \times \:  \Vert  a^{(1-z) \frac{p}{p_0} + z \frac{p}{p_1} }  \Vert_{L^2(\widehat{G})}  \Vert  b^{  (1-z) \frac{p'}{p_0'} + z \frac{p'}{p_1'} }  \Vert_{L^2(\widehat{G})}.
\end{split}
\end{equation}
We see that for any $z \in \mathcal{S}$ we have  $\vert F(z) \vert \leq e^{A e^{B t}}$ for suitable constants $A >0$ and $0 < B < \pi$.

\vspace{0.3cm}

\noindent {\it Step 1b: Estimating $F$ on  $i \mathbb{R}$.}
By \eqref{Eqn=MasterEstimate} and recalling that $p_0=p_0'=2$ we have in particular that
\[
\begin{split}
 \vert F(it) \vert \leq  &
    C_\beta C_\beta' (1 + \vert t \vert^3 \log^2(\vert t \vert) )^{\vert \frac{1}{q_\beta} - \frac{1}{2} \vert}  \Vert \OmegaA^s(m)     \Vert_{ L^{q}(K \backslash G \slash K) }^{\frac{q}{q_0}}  \Vert  a  \Vert_{L^p(\widehat{G})}^{\frac{ p }{2}}  \Vert  b   \Vert_{L^{p'}(\widehat{G})}^{\frac{ p' }{2}}.
\end{split}
\]

\vspace{0.3cm}

\noindent {\it Step 1c: Estimating $F$ on  $1+ i \mathbb{R}$.} We apply  Theorem \ref{Thm=Lp} to the symbol $M_{1 +it}$. This is possible as we have $m_s = \OmegaA^s(m) \in L^{2S_G/s}(K \backslash G \slash K)$ and therefore, recalling that $s_1 = S_G$,
\[
 m_s \vert m_s \vert^{\frac{-s   - it s_0 + (1 + it) s_1 }{s} } =  m_s \vert m_s \vert^{-1 + \frac{  it(s_1 - s_0)  + s_1 }{s} } \in L^{2}(K \backslash G \slash K).
\]
So that $M_{1 + it} = \OmegaA^{- it s_0 - (1+it) s_1 } (    m_s \vert m_s \vert^{-1 + \frac{  it(s_1 - s_0)  + s_1 }{s} }   )$ lies in $L^{2}(K \backslash G \slash K)$ as negative powers of $\OmegaA$ are bounded operators. Further in Step 1a we already justified that $M_{1+it}$  also lies in $L^{\infty}(K \backslash G \slash K)$. Hence we can apply  Theorem \ref{Thm=Lp}. Together with  Corollary \ref{Cor=L2CaseCasimir} it gives that there exists a constant $C_G > 0$ such that
\[
\begin{split}
 \vert F(1+it) \vert \leq  &  C_G \Vert \OmegaA^{s_1} ( M_{1+it} ) \Vert_{L^{2}( G)}  \Vert    a^{ -it \frac{p}{p_0} + (1+it) \frac{p}{p_1} }  \Vert_{L^{p_1}(\widehat{G})}  \Vert   b^{  it \frac{p'}{p_0'} + (1+it) \frac{p'}{p_1'} }  \Vert_{L^{p_1'}(\widehat{G})} \\
 = & C_G \Vert \OmegaA^{s_1} ( M_{1+it} ) \Vert_{L^{2}(G)}  \Vert    a \Vert_{L^{p}(\widehat{G})}^{  \frac{p}{p_1} }   \Vert   b  \Vert_{L^{p'}(\widehat{G})}^{    \frac{p'}{p_1'} }.
\end{split}
\]
Further, recalling that $s_1 = S_G$,
\[
\begin{split}
\Vert \OmegaA^{s_1} ( M_{1+it} ) \Vert_{L^{2}( G)}  = &
\Vert
\OmegaA^{  it ( s_1 - s_0)  }(  \OmegaA^s(m)  \vert\OmegaA^s(m)\vert^{-1 +  \frac{  it(s_1 - s_0)  + s_1 }{s}}  )  \Vert_{L^{2}(G) }   \\
= & \Vert      \vert\OmegaA^s(m)\vert^{ \frac{  s_1}{s}    } )  \Vert_{L^{2}( G )}
=   \Vert   \OmegaA^s(m)  \Vert_{L^{\frac{2S_G}{s}}( G )}^{\frac{S_G}{s}} =   \Vert   \OmegaA^s(m)  \Vert_{L^{q }( G )}^{\frac{q}{2}}.
\end{split}
\]

\vspace{0.3cm}

\noindent {\it Step 2: Remainder of the proof.} We apply Lemma \ref{Lem=ThreeLines}. The assumptions are met by Steps 1a, 1b and 1c. Further we find that
\[
\begin{split}
 \frac{2}{\sin(\pi \beta)} D_\beta \leq & \int_{-\infty}^\infty \frac{F_0(t)  +
   \frac{q}{q_0} \log(\Vert \OmegaA^s(m)     \Vert_{ L^{q}( G ) }) + \frac{p}{2} \log(   \Vert  a  \Vert_{L^p(\widehat{G})}) + \frac{p'}{2} \log(   \Vert  b   \Vert_{L^{p'}\widehat{G})} )
   }{ \cosh(\pi t) - \cos(\pi \beta) } dt \\
   &  +
\int_{-\infty}^\infty \frac{
   \frac{q}{2} \log(   \Vert   \OmegaA^s(m)  \Vert_{L^{q }(G)}  ) +    \frac{p}{p_1}   \log(\Vert    a \Vert_{L^{p}(\widehat{G})})  +  \frac{p'}{p'_1} \log( \Vert   b  \Vert_{L^{p'}(\widehat{G} )}   )
   }{ \cosh(\pi t) + \cos(\pi \beta) } dt
\end{split}
\]
where
\[
F_0(t) = \log(C_\beta C_\beta') + \left|  \frac{1}{q_\beta} - \frac{1}{2} \right|  \log(    1 + \vert t \vert^3 \log^2(\vert t \vert)       ).
\]
Recall from  \eqref{Eqn=Mtheta} that $M_\theta = m$. It follows by the remarks after Lemma \ref{Lem=ThreeLines} that for some constant  $C_{G, p, p_0,p_1} >0$ we have
\[
\tau_{\widehat{G}}( T_{m}(  a  )     b      ) =
\vert F(\theta) \vert \leq \exp(D_\theta) \leq C_{G, p, p_0,p_1} \Vert \OmegaA^s(m)     \Vert_{ L^{q}(  G  ) }  \Vert  a  \Vert_{L^p(\widehat{G})}
 \Vert  b   \Vert_{L^{p'}\widehat{G})}.
\]
Since this holds for all possible $a$ and $b$ as defined in the beginning of the proof a density argument concludes the proof.
\end{proof}

\section{Examples: multipliers with slow decay}\label{Sect=Examples}

The aim of this section is to illustrate that Theorem \ref{Thm=Main} provides new examples of Fourier multipliers on a wide class of Lie groups.
 To the knowledge of the author the only examples of $L^p$-multipliers on classes of semi-simple Lie groups come from the two papers \cite{PRS}, \cite{Tablate}.
 The method of proof \cite{PRS}, \cite{Tablate} is very effective to find bounds of $L^p$-multipliers for symbols $m$ that are supported on some neighbourhood of the identity of $G$. Then in \cite{PRS} a patching argument \cite[Proof of Theorem A]{PRS} is used to provide bounds of sums of translates of such multipliers. Due to this patching argument, or a simple and crude triangle inequality, the norms grow with the $L^1$-norm of such a multiplier. Hence these multipliers have a local behavior. The method was   improved upon in \cite{Tablate} yielding also multipliers without such an integrability property. Here we show that for $p$ closer to 2 even less conditions are needed and we get multipliers with an even slower decay rate, see Remark \ref{Rmk=Examples}.

\vspace{0.3cm}

For $f: \fraka \rightarrow [0,1]$ a $C^\infty$-function that is invariant under the action of the Weyl group  we define
\[
(\Psi f)( k_1 \exp(H) k_2 ) = f( H   ), \qquad H \in \fraka, k_1, k_2 \in K.
\]
The Weyl group invariance assures that this function is well-defined. Then $\Psi f$ is a bi-$K$-invariant smooth function on $G = K \exp(\fraka) K$.

 \begin{theorem}\label{Thm=Example}
 Let $ \frac{1}{4} \dim(G \slash K) < S_G < \frac{1}{2}\dim(G \slash K)$,  $s \in \mathbb{N}_{\geq 1} \cap [1, S_G]$ and $p \in (1, \infty)$ with $\vert \frac{1}{2} - \frac{1}{p} \vert < \frac{s}{2 S_G}$. Let $A_{dec} > \frac{ s \Vert \rho \Vert }{ S_G}$ and let $f: \fraka \rightarrow [0,1]$ be a $C^\infty$-function that is invariant under the Weyl group such that
  \begin{equation}\label{Eqn=fDecay}
 f( H ) = e^{-A_{dec} \Vert H \Vert}, \qquad \textrm{ in case } H \in \fraka, \Vert H \Vert \geq 1.
 \end{equation}
 Then
\begin{equation}\label{Eqn=Bound}
\Vert T_{\Psi f}: L^p(  \widehat{G} ) \rightarrow L^p(  \widehat{G} ) \Vert    < \infty.
\end{equation}
\end{theorem}
\begin{proof}
 By \cite[Proposition II.3.9]{HelgasonSpherical} there  exists a second order Weyl group invariant linear diffential operator $\mathcal{D}$ acting on $C^\infty( \fraka )$ such that $\OmegaA (\Psi f)   = \Psi ( \mathcal{D}  f)$ and therefore $\OmegaA^s(\Psi f) =  \Psi ( \mathcal{D}^s  f)$. We emphasize that $\mathcal{D}$  contains  differential operators of lower order as well.
  Then, by the explicit form of $f$ in \eqref{Eqn=fDecay} we have for $H \in \fraka$ with $\Vert H \Vert > 1$ that $\vert (\mathcal{D}^s f)(H) \vert \leq  D_s   e^{-A_{dec} \Vert H \Vert}$ for some constant $D_s >  0$. It follows that,
\begin{equation}\label{Eqn=LpComp}
\Vert \OmegaA^s (\Psi f) \Vert_{L^{2 S_G/s}(G) } = \Vert  \Psi ( \mathcal{D}^s  f)  \Vert_{L^{2 S_G / s}(G) } \leq
 D_s \vert K \slash M \vert
\left( \int_{\fraka^+}   e^{-\Vert H \Vert A_{dec} 2 S_G/s} \delta(H) dH \right)^{\frac{s}{2S_G}}.
\end{equation}
As by \eqref{Eqn=CartanIntegral2} we have  $\delta(H) \leq e^{-2\Vert H \Vert \Vert \rho \Vert}$ we see that this integral is finite by our choice of $A_{dec}$.
We now apply  Theorem \ref{Thm=Main}. We cannot do this directly as $\Psi f$ is not in $L^2(G)$ but we can use an approximation. Let $B_R := \{ H \in \fraka \mid \Vert H \Vert < R \}$ where $R>1$ and let $1_{B_R}$ be the indicator function on $B_R$. Set $f_R = f 1_R$. Let $\varphi_i: G  \rightarrow [0, \infty), i \in I$ be a net of symmetric bi-$K$-invariant $C^\infty$-functions with compact supports shrinking to $K$ and normalized by $\Vert \varphi_i \Vert_{L^1(G)} = 1$. Set $f_{R, i} = (\Psi f_R) \ast \varphi_i$.  Then  $f_{R,i} \in L^2(G)$ as $f_{R,i}$ has compact support and further $f_{R,i} \rightarrow \Psi f$ uniformly   as $R \rightarrow \infty$ and taking the limit in $i \in I$. Further, as $\OmegaA$ is a left-invariant differential operator we have $\OmegaA^s(    f_{R,i} ) = (\Psi f_R) \ast  \OmegaA^s( \varphi_i )$ and so certainly $f_{R,i} \in \Dom( \OmegaA^s)$ and $\OmegaA^s(f_{R,i}) \in L^{2 S_G/s}(G)$ as $\OmegaA^s(f_{R,i})$   has compact support.  It follows from  Theorem \ref{Thm=Main} that there is a constant $C_{G, s,p} >0$ such that
\begin{equation}\label{Eqn=Bound}
\begin{split}
\Vert T_{\Psi f}: L^p(  \widehat{G} ) \rightarrow L^p(  \widehat{G} ) \Vert \leq  &
\limsup_{i \in I} \limsup_{R \rightarrow \infty} \Vert T_{ f_{R,i} }: L^p(  \widehat{G} ) \rightarrow L^p(  \widehat{G} ) \Vert \\
\leq &
 C_{G, s, p} \limsup_{i \in I} \limsup_{R \rightarrow \infty}  \Vert \Omega_K^s(   f_{R,i}  )  \Vert_{L^{2 S_G/s}(G) }. \\
\end{split}
\end{equation}
We show that for each individual $i \in I$ the limsup over $R \rightarrow \infty$ of the latter expression is bounded by  $\Vert \OmegaA^s(   \Psi f  )  \Vert_{L^{2 S_G/s}(G) }$. Indeed, let $U_i$ be the symmetric compact support of $\varphi_i$ which is left and right $K$-invariant. The support of $\Omega^s(\varphi_i)$ is then still contained in $U_i$.  Let $S_R = \{ H \in \fraka \mid \Vert H \Vert = R\}$ and set $V_{R, i} := K \exp(S_R) K U_i = K \exp(S_R) U_i  \subseteq G$. Also set $V_{R,i}^+ = V_{R,i} U_i$.
 We have that,
\[
\begin{split}
\Vert \Omega^s( \Psi f_R \ast \varphi_i ) \Vert_{L^p(G)} \leq &
\Vert \Omega^s( \Psi f_R \ast \varphi_i ) 1_{G \backslash V_{R,i}} \Vert_{L^p(G)}  + \Vert \Omega^s( \Psi f_R \ast \varphi_i ) 1_{  V_{R,i}} \Vert_{L^p(G)} \\
 = &
\Vert (\Omega^s( \Psi f ) \ast \varphi_i ) 1_{G \backslash V_{R,i}} \Vert_{L^p(G)}  + \Vert ( (\Psi f_R) 1_{V_{R,i}^+} \ast  \Omega^s( \varphi_i ) ) 1_{  V_{R,i}}  \Vert_{L^p(G)} \\
\leq &
\Vert \Omega^s( \Psi f ) \Vert_{L^p(G) } \Vert \varphi_i \Vert_{L^1(G)}    + \Vert  (\Psi f_R) 1_{V_{R,i}^+} \Vert_{L^p(G)}  \Vert \Omega^s( \varphi_i )     \Vert_{L^1(G)} \\
\leq & \Vert \Omega^s( \Psi f ) \Vert_{L^p(G) }     + \Vert  (\Psi f) 1_{V_{R,i}^+} \Vert_{L^p(G)}  \Vert \Omega^s( \varphi_i )     \Vert_{L^1(G)}.
\end{split}
\]
Note that $\Psi f$ is in $L^p(G)$ and therefore it follows that $\lim_{R \rightarrow \infty} \Vert  (\Psi f) 1_{V_{R,i}^+} \Vert_{L^p(G)} = 0$ and so
\[
\limsup_{R \rightarrow \infty }\Vert \Omega^s( \Psi f_R \ast \varphi_i ) \Vert_{L^p(G)} \leq \Vert \Omega^s( \Psi f ) \Vert_{L^p(G) }.
\]
It follows that \eqref{Eqn=Bound} is finite.
 \end{proof}

\begin{remark}
In Theorem \ref{Thm=Example} the closer $p$ is to 2, the smaller we may take $s \in \mathbb{N}_{\geq 1}  \cap [1, S_G]$ and the slower the decay rate $A_{dec}$ of the symbol $m := \Psi f$ becomes.
\end{remark}

\begin{remark}
Suppose that $G = {\rm SL}(n, \mathbb{R})$ so that $K = {\rm SO}(n, \mathbb{R})$ and $A$ consists of diagonal matrices with trace 1. Then $\fraka$ are the diagonal matrices with trace 0. We have that
\[
\dim(G \slash K) = \frac{1}{2} n (n+1) - 1.
\]
And we recall that we typically chose $S_G = \frac{1}{4} \dim(G \slash K) + \epsilon$ for some $\epsilon > 0$ small. The Killing form is given by $\langle X, Y \rangle = 2n {\rm Tr}(XY)$ and $\rho = \sum_{1 \leq i < j \leq n}  \alpha_{i,j}$ where $\alpha_{i,j}(H) = H_{i} - H_{j}$ for $H = {\rm Diag}(H_1, \ldots, H_n) \in \fraka$  the diagonal trace 0 matrix with diagonal entries $H_i \in \mathbb{R}$. Then  $\alpha_{i,j} \in \fraka^\ast$ is identified with $\frac{1}{2n}(E_{ii} - E_{jj}) \in \fraka$ through the Killing form.
 Therefore,
\[
\Vert \rho \Vert^2 = \langle \rho, \rho \rangle =   \frac{1}{n} \left( \sum_{k=1}^n (n-k)^2 - \sum_{k=1}^n (n-k) (k-1)  \right) \approx
\frac{1}{n} \int_0^n  (n-x)(n-2x+1) dx =
\frac{1}{6} n (n+3).
\]
Therefore, for $n$ large enough
\[
\frac{\Vert \rho \Vert^2}{\dim(G \slash K)^2} \approx \frac{\frac{1}{6} n (n+3) }{ \frac{1}{4} n^2 (n+1)^2 } \approx \frac{2}{3 n^2}.
\]
Therefore, take $\frac{1}{4}  \dim(G \slash K) < S_G <  \frac{1}{2}  \dim(G \slash K)$,  $s \in \mathbb{N}_{\geq 1} \cap [1, S_G]$  and
 \[
 A_{dec} >  \frac{s \Vert \rho \Vert}{ S_G} \qquad \left( \approx 4 \sqrt{ \frac{2}{3} } \frac{s}{n} \textrm{ for } n \textrm{ large and } S_G \searrow \frac{1}{4}  \dim(G \slash K)\right),
     \]
 We have that for $f: \fraka \rightarrow \mathbb{C}$ smooth with
\[
 f( H ) = e^{- A_{dec}    \Vert H \Vert}, \qquad \textrm{ in case }    \Vert H \Vert  \geq 1.
\]
that $\Psi f$ is the symbol of an $L^p$-Fourier multiplier for $\vert \frac{1}{2} - \frac{1}{p} \vert < \frac{s}{2 S_G}$.
\end{remark}

\begin{remark}\label{Rmk=Examples}
Let us argue that the multipliers we have constructed here for $p$ close to 2 are new compared to what is known from the results in \cite{PRS} and \cite{Tablate}.
Remark \cite[Remark 3.8]{PRS} excludes the symbol $m = \Psi f$ we have constructed from the class of multipliers obtained in \cite[Theorem A]{PRS} as \cite[Remark 3.8]{PRS} implies that the symbols are integrable. Our symbols are not necessarily integrable as can easily be seen from \eqref{Eqn=CartanIntegral}. The multipliers we construct here are also out of reach of the theorem \cite[Theorem A2]{Tablate}. Indeed, assume $G = {\rm SL}(n, \mathbb{R}), n \geq 3$; if $n =2$ \cite[Theorem A2]{Tablate} is not applicable in the first place as the discussion following that theorem shows. Then for $H = {\rm Diag}(H_1, \ldots, H_n) \in \fraka$ set $\Vert H \Vert_\infty = \max_{1 \leq i \leq n} \vert H_i \vert$ and $\Vert H \Vert_2 = {\rm Tr}(H^2)^{\frac{1}{2}} = (2n)^{- \frac{1}{2}} \Vert H \Vert$. We have (see \cite{Tablate} for the adjoint representation and its norm),
\[
\Vert {\rm Ad}_{\exp(H)} \Vert \geq \exp(  \Vert H \Vert_\infty) \geq  \exp( n^{-\frac{1}{2}}  \Vert H \Vert_{2} ) \geq   \exp( 2^{-1/2} n^{-1}  \Vert H \Vert ).
\]
So that \cite[Equation following Theorem A2]{Tablate} yields that
\[
\vert (\Psi f)(\exp( H )) \vert \preceq \Vert {\rm Ad}_{\exp(H)} \Vert^{-d_G} \leq   \exp(-  2^{-1/2} d_G n^{-1}  \Vert H \Vert ),
\]
where $d_G = \lfloor n^2/4 \rfloor$ (see \cite[Example 3.14]{CJKM}, \cite{Maucourant}).
So as $d_G$ increases as $n$ increases we see that the decay of these multipliers is faster than in our examples.  Of course \cite{Tablate} considers symbols that are multipliers in the full range $p \in (1,\infty)$ whereas our methods yield the sharper estimates only when $p$ approximates 2.
\end{remark}

\end{document}